\documentclass[11pt,a4paper]{article}
\usepackage[T1]{fontenc}
\usepackage{graphicx}
\usepackage{mathtools}
\usepackage{amssymb}
\usepackage{amsthm}
\usepackage{thmtools}
\usepackage{microtype}
\usepackage{ragged2e}
\justifying 
\usepackage{xcolor}
\usepackage{nameref}
\usepackage{hyperref}
\hypersetup{
	colorlinks=true,
	linkcolor=blue,
	citecolor=blue,
	filecolor=magenta,      
	urlcolor=cyan,
}
\usepackage{algorithm}
\usepackage{algorithmicx}
\usepackage{algpseudocode} 
\usepackage{setspace}
\usepackage{amsfonts}
\title{Pontryagin's Principle Based Algorithms for Optimal Control Problems of Parabolic Equation}
\author{Weilong You \thanks{College of Science, University of Shanghai for Science and Technology,
		Shanghai 200093, China (email: {\tt 1127880510@qq.com}).}
	\and Fu~Zhang\thanks{Corresponding author. College of Science, University of Shanghai for Science and Technology,
		Shanghai 200093, China (email: {\tt fuzhang82@gmail.com}).
		F.~Zhang was partially supported 
		by the National Natural Science Foundation of China (No. 11701369, 12071292).}}
\date{\today}
\allowdisplaybreaks
\newtheorem{theorem}{Theorem}[section] % Theorem 环境，编号按章节
\newtheorem{lemma}[theorem]{Lemma}     % Lemma 环境，与 Theorem 共享编号
\newtheorem{proposition}[theorem]{Proposition} % Proposition 环境，与 Theorem 共享编号

\newtheorem{assumption}{Assumption}
\begin{document}
	\maketitle
	\textbf{Abstract} This paper applies the Method of Successive Approximations (MSA) based on Pontryagin's principle to solve optimal control problems with state constraints for semilinear parabolic equations. Error estimates for the first and second derivatives of the function are derived under \( L^{\infty} \)-bounded conditions. An augmented MSA is developed using the augmented Lagrangian method, and its convergence is proven. The effectiveness of the proposed method is demonstrated through numerical experiments.
	\section{Introduction}
In this paper, we investigate the optimal control problem for parabolic equations with state constraints and Neumann boundary conditions, formulated as follows:

\begin{equation}
	\begin{aligned}
		\min_{u,v} J(y, u, v) := &\int_{\Omega_T} F(x, t, y, u) \, dx \, dt + \int_{\Sigma_T} G(s, t, y, v) \, ds \, dt \\
		&+ \int_{\Omega} L(x, y(x, T)) \, dx,
	\end{aligned}
	\label{P}
	\tag{P}
\end{equation}
subject to the dynamics
\begin{align*}
	\frac{\partial y}{\partial t} + A y + f(x, t, y, u) &= 0, \quad \mathrm{in }\; \Omega_T, \\
	\frac{\partial y}{\partial n_A} + v &= 0, \quad \mathrm{on }\; \Sigma_T, \\
	y(\cdot, 0) &= y_0, \quad \mathrm{in }\; \overline{\Omega},
\end{align*}
where \( A \) is a second-order elliptic operator. The specific problem setup is described in Subsection 2.1.

The solution methods most commonly employed to address such problems include solving the Hamilton-Jacobi-Bellman (HJB) equation \cite{Bardi1997O,Falcone2014S} and utilizing the Pontryagin maximum principle \cite{Casas1997P,Raymond1999H}. Although both approaches have been extensively studied, analytical solutions are generally unavailable for most practical cases, leading to the development of various numerical methods.

Current numerical techniques for solving the HJB equation include semi-Lagrangian schemes \cite{Falcone2014S}, sparse grid approaches \cite{Kang2017M}, tensor-calculus-based methods \cite{Stefansson2016S}, and successive Galerkin approximations \cite{Beard1997G,Beard1998A,Beeler2000F,Kalise2018P}. Notably, policy iteration and Howard's algorithm \cite{Alla2015A,Bokanowski2009S,Howard1960D} are discrete implementations of the continuous Galerkin method. In the numerical solution of Pontryagin’s optimality principle, approaches such as the two-point boundary value problem method \cite{Bryson1975A,Roberts1972T}, collocation methods combined with nonlinear programming techniques \cite{Betts1994A,Betts1997E,Bertsekas1999N,Bazaraa2006N}, and the Method of Successive Approximations (MSA) \cite{Chernous1982M} are often applied.

MSA plays a crucial role in the numerical resolution of both types of problems, as it is an iterative procedure based on alternating propagation and optimization steps. In \cite{Li2017M}, MSA is employed to solve optimal control problems governed by ordinary differential equations, with an extended version derived via the augmented Lagrangian method \cite{Hestenes1969M}. Additionally, in \cite{Sethi2024T}, a modified version of MSA for stochastic control problems is proposed, and convergence of the algorithm is established. In this paper, we provide an error estimate for MSA applied to optimal control problems of parabolic equations under specific conditions (see Lemma 3.3) and propose the corresponding augmented MSA.

The structure of this paper is as follows: In Section 2, we present preliminary results related to the original problem \(\eqref{P}\) and the corresponding Pontryagin maximum principle. In Section 3, we apply the Method of Successive Approximations (MSA) and its improved variant to solve the problem. Specifically, in Section 3.1, we introduce the basic MSA based on the Pontryagin principle; in Section 3.2, we state certain assumptions under which we derive the error estimate for the basic MSA; in Section 3.3, we introduce the augmented Pontryagin principle and the augmented MSA based on it; and in Section 3.4, we prove the convergence of the augmented MSA. Finally, in Section 4, we demonstrate the effectiveness of the augmented MSA through numerical experiments.\\
\textbf{Notation} Throughout this paper, \( (\cdot, \cdot)_{\Omega} \) denotes the inner product in \( L^2(\Omega) \), and \( (\cdot, \cdot)_{\Gamma} \) denotes the inner product in \( L^2(\Gamma) \). The constant \( C \) appearing in the text may represent different values and will be adjusted as necessary.

	\section{Preliminary results}
   \subsection{Problem setting}
Let \( \Omega \subset \mathbb{R}^N \) be an open, bounded subset of \( \mathbb{R}^N \) with a \( C^{2,\beta} \)-boundary, where \( N \in \{2,3\} \). That is, the boundary \( \partial \Omega \) is a \( (N-1) \)-dimensional manifold of class \( C^{2,\beta} \), and \( \Omega \) lies locally on one side of its boundary. A function is said to be of class \( C^{2,\beta} \) if it is of class \( C^2 \) and its second-order partial derivatives are Hölder continuous with exponent \( \beta \). For a fixed time interval \( T > 0 \), we define the space-time domain \( \Omega_T := \Omega \times (0,T) \) and its lateral boundary \( \Sigma_T := \partial \Omega \times (0,T) \).

 The function space $\mathcal{Y}$ is given by
 \[
 \mathcal{Y} := \mathcal{W}(0,T; L^2(\Omega), H^1(\Omega)) \cap C(\overline{\Omega}_T),
 \]
 where $\mathcal{W}(0,T; L^2(\Omega), H^1(\Omega))$ denotes the Hilbert space of functions $y$ such that $y \in L^2(0,T; H^1(\Omega))$ and $y_t \in L^2(0,T; H^{-1}(\Omega))$. The associated norm is defined by
 \[
 \|y\|_{\mathcal{W}(0,T; L^2(\Omega), H^1(\Omega))} := \sqrt{\|y\|_{L^2(0,T; H^1(\Omega))}^2 + \|y_t\|_{L^2(0,T; H^{-1}(\Omega))}^2}.
 \]
 
 The control spaces are defined as $\mathcal{U} := L^r(\Omega_T)$ and $\mathcal{V} := L^q(\Sigma_T)$, where the exponents $r > \frac{N}{2} + 1$ and $q > N + 1$ ensure the required regularity and integrability conditions.
 
 Consider the second-order differential operator $A$ given by
 \[
 A y(x, t) = -\sum_{i, j=1}^N \partial_{x_j} \big(a_{ij}(x) \partial_{x_i} y(x, t)\big),
 \]
 where the coefficients $a_{ij}(x) \in L^\infty(\Omega)$ satisfy $a_{ij}(x) = a_{ji}(x)$ for all $i, j$. The operator $A$ is assumed to be uniformly elliptic, meaning there exists a constant $K > 0$ such that
 \[
 \sum_{i, j=1}^N a_{ij}(x) \xi_i \xi_j \geq K \|\xi\|^2 \quad \text{for almost every } x \in \Omega \text{ and all } \xi \in \mathbb{R}^N.
 \]
 
 The goal is to minimize the cost functional
 
 \begin{equation}
 \begin{aligned}
 	J(y, u, v) := &\int_{\Omega_T} F(x, t, y, u) \, dx \, dt + \int_{\Sigma_T} G(s, t, y, v) \, ds \, dt \\
 	&+ \int_{\Omega} L(x, y(x, T)) \, dx,
 \end{aligned}
 \label{1}
 \end{equation}
 
 subject to the parabolic system
 \begin{equation}
 \begin{aligned}
 	\frac{\partial y}{\partial t} + A y + f(x, t, y, u) &= 0, \quad \mathrm{in }\; \Omega_T, \\
 	\frac{\partial y}{\partial n_A} + v &= 0, \quad \mathrm{on }\; \Sigma_T, \\
 	y(\cdot, 0) &= y_0, \quad \mathrm{in }\; \overline{\Omega},
 \end{aligned}
 \label{2}
\end{equation}
 where the conormal derivative is given by
 \[
 \frac{\partial y}{\partial n_A}(x, t) = \sum_{i, j=1}^N a_{ij}(x) \partial_{x_i} y(x, t) \nu_j(x),
 \]
 and $\nu = (\nu_1, \dots, \nu_N)$ denotes the unit outward normal vector to $\Gamma$.\\
The following assumption applies throughout this paper, and the functions \( f \), \( F \), \( G \), and \( L \) are assumed to satisfy the conditions below.  

\begin{assumption}  
	Let \( M_1(\cdot,\cdot) \in L^r(\Omega_T) \), \( M_2(\cdot) \in L^1(\Omega) \), \( M_3 > 0 \), \( M_4(\cdot,\cdot) \in L^1(\Omega_T) \), \( M_5(\cdot,\cdot) \in L^1(\Sigma_T) \), \(M_6(\cdot,\cdot)\in L^q(\Sigma_T)\) and \( \eta(|\cdot|) \) be a nondecreasing function mapping \( \mathbb{R}^+ \) to \( \mathbb{R}^+ \). Furthermore, let \( m_1 > 0 \) and \( C_0 > 0 \). The following conditions hold:  
	
	1. The function \( f(x,t,y,u) \):  
	- For every \( (y,u) \in \mathbb{R}^2 \), \( f( \cdot, \cdot, y, u) \) is measurable on \( \Omega_T \).  
	- For almost every \( (x,t) \in \Omega_T \), \( f(x,t,\cdot,\cdot) \) is continuous on \( \mathbb{R} \times \mathbb{R} \). 
	- For almost every \( (x,t) \in \Omega_T \) and every \(u \in \mathbb{R}\), \(f(x,t,\cdot,u)\) is of class \( C^2 \) on \(\mathbb{R}\).  
	- For almost every \( (x,t) \in \Omega_T \),  
	\[
	|f(x,t,0,u)| \leq M_1(x,t) + m_1 |u|, \]
	\[\quad C_0 \leq f'_y(x,t,y,u) \leq (M_1(x,t) + m_1 |u|)\eta(|y|).
	\]  
	
	2. The function \( L(x,y) \):  
	- For every \( y \in \mathbb{R} \), \( L(\cdot, y) \) is measurable on \( \Omega \).  
	- For almost every \( x \in \Omega \), \( L(x,\cdot) \) is continuous on \( \mathbb{R} \) and is of class \( C^2 \) on \(\mathbb{R}\).  
	- For almost every \( x \in \Omega \),  
	\[
	|L(x,0)| \leq M_2(x), \quad |L'_y(x,y)| \leq M_3 \eta(|y|),\]
	\[ \quad |L'_y(x,y) - L'_y(x,z)| \leq M_3 \eta(|y|)\eta(|z|).
	\]  
	
	3. The function \( F(x,t,y,u) \):  
	- For every \( (y,u) \in \mathbb{R}^2 \), \( F(\cdot,\cdot, y, u) \) is measurable on \( \Omega_T \).  
	- For almost every \( (x,t) \in \Omega_T \), \( F(x,t,\cdot,\cdot) \) is continuous on \( \mathbb{R} \times \mathbb{R} \). 
	- For almost every \( (x,t) \in \Omega_T \) and every \(u \in \mathbb{R}\), \(F(x,t,\cdot,u)\) is of class \( C^2 \) on \(\mathbb{R}\). 
	- For almost every \( (x,t) \in \Omega_T \),  
	\[
	|F(x,t,0,u)| \leq M_4(x,t) + m_1 |u|^r, \quad |F'_y(x,t,y,u)| \leq (M_1(x,t) + m_1 |u|)\eta(|y|).
	\]  
	
	4. The function \( G(s,t,y,v) \):  
	- For every \( (y,v) \in \mathbb{R}^2 \), \( G(\cdot,\cdot, y, v) \) is measurable on \( \Sigma_T \).  
	- For almost every \( (s,t) \in \Sigma_T \), \( G(s,t,\cdot,\cdot) \) is continuous on \( \mathbb{R} \times \mathbb{R} \). 	
	- For almost every \( (s,t) \in \Sigma_T \) and every \(v \in \mathbb{R}\), \(G(s,t,\cdot,v)\) is of class \( C^2 \) on \(\mathbb{R}\). 
	- For almost every \( (s,t) \in \Sigma_T \),  
	\[
	|G(s,t,0,v)| \leq M_5(s,t) + m_1 |v|^q, \quad |G'_y(s,t,y,v)| \leq (M_6(s,t) + m_1 |v|)\eta(|y|).
	\]  
\end{assumption}
   \subsection{Pontryagin's Principle}
In this section, we will introduce the first order necessity condition for the optimal solution of the problem p, which is the Pontryagin's principle. Among them, the complete proof process of the Pontryagin's principle of the control problem governed by Semilinear parabolic equation involved in this paper is given in \cite{Raymond1999H}.\\
First, we define the Hamiltonian function expression of the problem \eqref{P}, including the distribution Hamiltonian function and the boundary Hamiltonian function, and the expression is as follows:\\
$$H_{\Omega_T}(x,t,y,u,p):=F(x, t, y, u)-pf(x, t, y, u),$$
for every $(x,t,y,u,p)\in \Omega \times (0,T)\times R\times R\times R,$\\
$$H_{\Sigma_T}(x,t,y,v,p):=G(s, t, y, u)-pv,$$
for every $(s,t,y,v,p)\in \Sigma \times (0,T)\times R\times R\times R$.\\
\begin{theorem}
	Let $(\bar{y},\bar{u},\bar{v})$ be the solution of $(P)$. Then there exists a unique adjoint state $\bar{p} \in \mathcal{W}(0,T;L^2(\Omega),H^{1}(\Omega))\cap C(\Omega_T) $, satisfying the following equation.\\
	\begin{equation}
	\begin{aligned}
		-\frac{\partial p}{\partial t} + Ap + f'_y(x,t,\bar{y},\bar{u})p &= F'_y(x,t,\bar{y},\bar{u}),  & \mathrm{in} \; \Omega_T,\\
		\frac{\partial p}{\partial n_A} &= G'_y(s,t,\bar{y},\bar{v}), & \mathrm{on}\; \Sigma_T,\\
		p(\cdot,T) &= L'_y(x, y(\cdot, T)), & \mathrm{in} \;\;\;\Omega.
	\end{aligned}
	\label{3}
		\end{equation}
	and such that\\
	\begin{align*}
		&H_{\Omega_T}(x,t,\bar{y},\bar{u},\bar{p})=\min_{u\in U_{ad}}H_{\Omega_T}(x,t,\bar{y},u,\bar{p}),\\
		&\mathrm{for \; a.e.}  \; (x,t)\in \Omega_T,\\
		&H_{\Sigma_T}(s,t,\bar{y},\bar{v},\bar{p})=\min _{v\in V_{ad}}H_{\Sigma_T}(s,t,\bar{y},v,\bar{p}),\\
		&\mathrm{for \; a.e.} \; (s,t)\in \Sigma_T.
	\end{align*}
\end{theorem}
\begin{proof}
	See Theorem 2.1 of \cite{Raymond1999H}.
\end{proof}
   \section{Method of Successive Approximations}
In this section, we propose a numerical algorithm to solve problem \eqref{P} based on the maximum principle. The algorithm's error analysis and convergence properties will be examined in the framework of continuous time and space. To achieve this, we employ the Method of Successive Approximations (MSA), initially introduced by Chernousko and Lyubushin in~\cite{Chernous1982M}. Below, we present the fundamental formulation of the method.
   \subsection{Basic MSA}
The basic MSA alternates between solving the state equation and the adjoint equation, followed by an optimization step to update the control variables. In this approach, the state equation captures the system's dynamics, while the adjoint equation provides gradient information essential for optimizing the Hamiltonian. At each iteration, the control variables are updated by minimizing the Hamiltonian over the admissible control sets. This iterative process is repeated until a predefined termination criterion is met, ensuring convergence under suitable conditions. The detailed steps of the algorithm are presented in Algorithm \ref{alg:msa} below.
\begin{algorithm}[H]
	\caption{Basic MSA}
	\label{alg:msa}
	\begin{algorithmic}[1]
		\State \textbf{Initialize:} Choose initial controls $(u_{1}, v_{1}) \in L^r(\Omega_T) \times L^q(\Sigma_T)$, set iteration index $i = 1$. 
		\Repeat 
		\State \textbf{Solve the state equation:}
		\[
		\frac{\partial y_i}{\partial t} + A y_i + f(x, t, y_i, u_i) = 0, \quad 
		\frac{\partial y_i}{\partial n_A} + v_i = 0, \quad y_i(\cdot,0) = y_0.
		\]
		\State \textbf{Solve the adjoint equation:}
		\[
		-\frac{\partial p_i}{\partial t} + A p_i = \nabla_y H_{\Omega_T}(x, t, y_i, u_i, p_i),
		\]
		\[
		\frac{\partial p_i}{\partial n_A} = \nabla_y H_{\Sigma_T}(s, t, y_i, v_i, p_i), \quad 
		p_i(\cdot, T) = L'_y(x, y_i(\cdot, T)).
		\]
		\State \textbf{Update the controls:}
		\[
		\begin{aligned}
			u_{i+1} &= \arg\min_{u \in U_{ad}} H_{\Omega_T}(x, t, y_i, u, p_i), \\
			v_{i+1} &= \arg\min_{v \in V_{ad}} H_{\Sigma_T}(s, t, y_i, v, p_i),
		\end{aligned}
		\]
		for each $(x, s, t) \in \Omega \times \Gamma \times [0,T]$.
		\State Increment the iteration index: $i := i + 1$.
		\Until{The termination criterion is satisfied.}
	\end{algorithmic}
\end{algorithm}
The convergence of the basic MSA has been established for a limited class of linear quadratic regulators~\cite{Aleksandrov1968O}. However, it is well-documented that the method often diverges in more general settings, particularly when unfavorable initial controls is chosen~\cite{Aleksandrov1968O, Chernous1982M}. This divergence underscores the need to understand the instability of the algorithm, specifically the relationship between the maximization step in Algorithm~\ref{alg:msa} and the underlying optimization problem \eqref{P}. To address this issue, our goal is to modify the basic MSA to ensure robust convergence, even under less favorable conditions.
   \subsection{Error Estimate}
In this section, we establish the relationship between the control index $J$ and the minimization of the Hamiltonian through a lemma. Prior to this, we introduce the necessary assumptions.
\begin{assumption}
For every \(y\in\mathcal{Y}\), \(u\in\mathcal{U}\), \(v\in\mathcal{V}\), the first-order and second-order derivatives of \( f \), \(L\), \(F\), \(G\) satisfy:  
\[
\parallel f'_y(\cdot,\cdot,y,u)\parallel_{L^{\infty}(\Omega_T)} \leq C, \quad \parallel f''_{yy}(\cdot,\cdot,y,u)\parallel_{L^{\infty}(\Omega_T)} \leq C,
\]  
\[
\parallel F'_y(\cdot,\cdot,y,u)\parallel_{L^{\infty}(\Omega_T)} \leq C, \quad \parallel F''_{yy}(\cdot,\cdot,y,u)\parallel_{L^{\infty}(\Omega_T)} \leq C,
\]  
\[
\parallel L'_y(\cdot,y)\parallel_{L^{\infty}(\Omega)} \leq C, \quad \parallel L''_{yy}(\cdot,y)\parallel_{L^{\infty}(\Omega)} \leq C,
\]  
\[
\parallel G'_y(\cdot,\cdot,y,v)\parallel_{L^{\infty}(\Sigma_T)} \leq C, \quad \parallel G''_{yy}(\cdot,\cdot,y,v)\parallel_{L^{\infty}(\Omega_T)} \leq C,
\]  
where \( C> 0 \) is a constant.  
\end{assumption}
Before deriving the error estimate, we first establish upper bounds for \( p^{\theta} \) and \( \delta y \) under their respective norms. These bounds are obtained through the following two lemmas. Here, we denote \( y^{\theta} := y(u^{\theta}, v^{\theta}) \), \( y^{\phi} := y(u^{\phi}, v^{\phi}) \), \( \delta y := y^{\phi} - y^{\theta} \), and \( p^{\theta} := p(u^{\theta}, v^{\theta}, y^{\theta}) \).
\begin{lemma}
Let Assumption 1 fulfilled. Then, there exists a constant $C_1$ such that\\ 
\begin{equation}
\begin{aligned}
	\parallel p^{\theta}\parallel_{L^{\infty}(\Omega_T)}\leq C_1.
\end{aligned}
\label{4}
\end{equation}
\end{lemma}
\begin{proof}
Refer to Proposition 4.1 of \cite{Raymond1999H}, along with the embedding \( L^r(\Omega_T) \hookrightarrow L^{\infty}(\Omega_T) \) and \( L^s(\Sigma_T) \hookrightarrow L^{\infty}(\Sigma_T) \), to obtain the result.
\end{proof}
   \begin{lemma}
   	Let Assumption 1 and Assumption 2 be fulfilled. Then, there exists constants $C_2>0$ such that\\ 
   	\begin{equation}
   	 \begin{aligned}
   	 	&\parallel\delta y\parallel^2_{L^2(\Omega_T)}+\parallel\delta y\parallel^2_{L^2(\Sigma_T)}\leq C_2\big(  \parallel v^{\phi}-v^{\theta}\parallel^2_{L^2(\Sigma_T)}\\
   	 	&+\parallel f(\cdot,\cdot, y^{\theta}, u^{\phi})-f(\cdot,\cdot, y^{\theta}, u^{\theta})\parallel^2_{L^2(\Omega_T)}\big).
   	 \end{aligned}
   	 \label{5}
   	    	\end{equation}
   \end{lemma}
   \begin{proof}By multiplying both sides of the parabolic equation \eqref{2} corresponding to \( \delta y \) by \( \delta y \) and integrating over \( \Omega \), we obtain 
   	\begin{align*}
   	&(\frac{\partial \delta y}{\partial t},\delta y)_\Omega + (A\delta y,\delta y)_\Omega + (f(\cdot, \cdot, y^{\phi}, u^{\phi})-f(\cdot, \cdot, y^{\theta}, u^{\theta}),\delta y)_\Omega=0,\\
    &(\frac{\partial \delta y}{\partial t},\delta y)_\Omega - \int_\Omega \sum_{i, j=1}^N \partial_{x_j} \big(a_{ij}(x) \partial_{x_i} \delta y\big) \delta y dx\\ 
    &+ (f(\cdot, \cdot, y^{\phi}, u^{\phi})-f(\cdot, \cdot, y^{\theta}, u^{\theta}),\delta y)_\Omega=0,\\
    &(\frac{\partial \delta y}{\partial t},\delta y)_\Omega = -(v^{\phi}-v^{\theta}, \delta y)_\Gamma-\int_\Omega \sum_{i, j=1}^N  a_{ij}(x) \partial_{x_i} \delta y \partial_{x_j}\delta y dx \\
    &-(f(\cdot, \cdot, y^{\phi}, u^{\phi})-f(\cdot, \cdot, y^{\theta}, u^{\theta}),\delta y)_\Omega.    
   	\end{align*}
   	Since $A$ is a uniform elliptic operator, there is\\
   	\begin{align*}
   	\int_\Omega \sum_{i, j=1}^N  a_{ij}(x) \partial_{x_i} \delta y \partial_{x_j}\delta y dx \geq K \parallel \nabla_x \delta y \parallel^2_{L^2(\Omega)}.
   	\end{align*}
   	Therefore
   	\begin{equation}
   	\begin{aligned}
   		&(\frac{\partial \delta y}{\partial t},\delta y)_\Omega+K \parallel \nabla_x \delta y \parallel^2_{L^2(\Omega)} \\
   		&\leq -(v^{\phi}-v^{\theta}, \delta y)_\Gamma -(f(\cdot, \cdot, y^{\phi}, u^{\phi})-f(\cdot, \cdot, y^{\theta}, u^{\theta}),\delta y)_\Omega, 
   	\end{aligned}
   	\label{6}
   	   	\end{equation}
   	where we used Young's inequality for \eqref{5} and $\epsilon < \frac{K}{2C}$, we get
   	\begin{align*}
   		&(\frac{\partial \delta y}{\partial t},\delta y)_\Omega+K \parallel \delta y \parallel^2_{H^1(\Omega)} \leq \frac{1}{\epsilon}\parallel v^{\phi}-v^{\theta}\parallel^2_{L^2(\Gamma)}+\epsilon\parallel \delta y\parallel^2_{L^2(\Gamma)} \\
   		&+\frac{1}{2}\parallel f(\cdot, \cdot, y^{\phi}, u^{\phi})-f(\cdot, \cdot, y^{\theta}, u^{\theta})\parallel^2_{L^2(\Omega)}+(\frac{1}{2}+K)\parallel \delta y\parallel^2_{L^2(\Omega)}.
   		   	\end{align*}
   		   	From the trace embedding theorem, we know $H^1(\Omega)\hookrightarrow L^2(\Gamma)$
   		   	\begin{align*}
   		   		\parallel \delta y\parallel^2_{L^2(\Gamma)}\leq C\parallel \delta y\parallel^2_{H^1(\Omega)},
   		   	\end{align*}
   		 which implies
   		   	\begin{align*}
   		&\frac{\partial}{\partial t}\parallel\delta y\parallel^2_{L^2(\Omega)}+K \parallel \delta y \parallel^2_{H^1(\Omega)} \leq\frac{2}{\epsilon}\parallel v^{\phi}-v^{\theta}\parallel^2_{L^2(\Gamma)}\\
   		&+\parallel f(\cdot, \cdot, y^{\phi}, u^{\phi})-f(\cdot,\cdot, y^{\theta}, u^{\theta})\parallel^2_{L^2(\Omega)}+(1+2K)\parallel \delta y\parallel^2_{L^2(\Omega)}\\
   		& \leq\frac{2}{\epsilon}\parallel v^{\phi}-v^{\theta}\parallel^2_{L^2(\Gamma)}
   		+\parallel f(\cdot, \cdot, y^{\theta}, u^{\phi})-f(\cdot, \cdot, y^{\theta}, u^{\theta})\parallel^2_{L^2(\Omega)}\\
   		&+(1+2K+C)\parallel \delta y\parallel^2_{L^2(\Omega)}.
   	\end{align*}
   	Thus the differential form of Gronwall's inequality yields the following estimate
   	\begin{align*}
    &\max_{0\leq t\leq T}\parallel\delta y\parallel^2_{L^2(\Omega)} +C\parallel \delta y \parallel^2_{L^2(0,T;H^1(\Omega))} \leq C\parallel v^{\phi}-v^{\theta}\parallel^2_{L^2(\Sigma_T)}\\
    &+C\parallel f(\cdot, \cdot, y^{\theta}, u^{\phi})-f(\cdot, \cdot, y^{\theta}, u^{\theta})\parallel^2_{L^2(\Omega_T)}.
   \end{align*}
   Since the embedding $L^{\infty}(0,T,L^2(\Omega))\hookrightarrow L^2(\Omega_T)$
   \begin{align*}
\parallel\delta y\parallel^2_{L^2(\Omega_T)} \leq C\max_{0\leq t\leq T}\parallel\delta y\parallel^2_{L^2(\Omega)},
   \end{align*}
   we proves
   \begin{align*}
   	&\parallel\delta y\parallel^2_{L^2(\Omega_T)}+\parallel\delta y\parallel^2_{L^2(\Sigma_T)}\leq C_2\big(\parallel v^{\phi}-v^{\theta}\parallel^2_{L^2(\Sigma_T)}\\
   	&+\parallel f(\cdot, \cdot, y^{\theta}, u^{\phi})-f(\cdot, \cdot,  y^{\theta}, u^{\theta})\parallel^2_{L^2(\Omega_T)}\big).
   \end{align*}
   \end{proof}
    \begin{lemma}
    	Suppose Assumption 1 and Assumption 2 holds. Then there exists a constant $\tilde{C}>0$ such that for any control variables $u^{\phi},u^{\theta}\in \mathcal{U}$, $v^{\phi},v^{\theta}\in \mathcal{V}$,
    	\begin{equation}
        		\begin{aligned}
    	J(u^{\phi},v^{\phi})-J(u^{\theta},v^{\theta})&\leq\int_{\Omega_T}[H_{\Omega_T}(y^{\theta}, u^{\phi})-H_{\Omega_T}(y^{\theta}, u^{\theta})]dxdt\\
    	&+\int_{\Sigma_T}[H_{\Sigma_T}(y^{\theta}, u^{\phi})-H_{\Sigma_T}(y^{\theta}, u^{\theta})]dsdt\\
    	&+\tilde{C}\big(\parallel u^{\phi}-u^{\theta}\parallel^2_{L^2(\Omega_T)}+\parallel v^{\phi}-v^{\theta}\parallel^2_{L^2(\Sigma_T)}\big).
    \end{aligned}
    \label{7}
\end{equation}
    \end{lemma}
    \begin{proof}By considering the difference \( J(u^{\phi},v^{\phi}) - J(u^{\theta},v^{\theta}) \), the right-hand side of the equation below can be divided into three parts
    	\begin{equation}
    	\begin{aligned}
    	    &J(u^{\phi},v^{\phi})-J(u^{\theta},v^{\theta})=\underbrace{\int_{\Omega_T} F(y^{\phi}, u^{\phi})-F(y^{\theta}, u^{\theta})dxdt}_{(A)}\\ 
    	    &+ \underbrace{\int_{\Sigma_T}G(y^{\phi}, v^{\phi})-G(y^{\theta}, v^{\theta})dsdt}_{(B)} 
    	    + \underbrace{\int_{\Omega} L(y^{\phi}(T))- L(y^{\theta}(T))dx}_{(C)}.
    	\end{aligned}
    	\label{8}
    		\end{equation}
    The third part \( (C) \) can be estimated using Hölder's inequality yielding
    	\begin{align*}
    	(C) &= \int_{\Omega}L'(y^{\theta}(T))\delta y(T) +\frac{1}{2}L'_{yy}(\xi_1)(\delta y(T))^2dx\\
    	&\leq \int_{\Omega}p^{\theta}(T)\delta y(T)dx+C\parallel\delta y\parallel^2_{L^2(\Omega_T)}\\
    	&\leq \underbrace{\int_{\Omega_T}\frac{\partial}{\partial t}(p^{\theta}\delta y)dxdt}_{(D)}+C\parallel\delta y\parallel^2_{L^2(\Omega_T)}.
    	\end{align*}
    		Adding \( A \) and \( D \), and utilizing \eqref{2}, \eqref{3}, along with the Taylor expansion, we obtain
    	\begin{align*}
    		(A)+(D)&=\int_{\Omega_T}F(y^{\phi}, u^{\phi})-F(y^{\theta}, u^{\theta})+\frac{\partial}{\partial t}(p^{\theta}\delta y)dxdt\\
    		&=\int_{\Omega_T}F(y^{\theta}, u^{\phi})-F(y^{\theta}, u^{\theta})+F'_y(y^{\theta}, u^{\phi})\delta y\\
    		&+\frac{1}{2}F'_{yy}(\xi_2,u^{\phi})(\delta y)^2+\frac{\partial}{\partial t}p^{\theta}\delta y+p^{\theta}\frac{\partial}{\partial t}\delta ydxdt\\
    		&\leq\int_{\Omega_T}F(y^{\theta}, u^{\phi})-F(y^{\theta}, u^{\theta})+F'_y(y^{\theta}, u^{\phi})\delta y\\
    		&+ f'_y(y^{\theta},u^{\theta})p^{\theta}\delta y- F'_y(y^{\theta}, u^{\theta})\delta y-p^{\theta}[f(y^{\phi}, u^{\phi})-f(y^{\theta},u^{\theta})]\\
    		&-\sum_{i, j=1}^N \partial_{x_j} \big(a_{ij}(x) \partial_{x_i} p^{\theta}\big) \delta y+\sum_{i, j=1}^N \partial_{x_j} \big(a_{ij}(x) \partial_{x_i}  \delta y\big)p^{\theta}\\
    		&+\frac{1}{2}F'_{yy}(\xi_2,u^{\phi})(\delta y)^2dxdt\\
    		&\leq \underbrace{-\int_{\Sigma_T}G'_y(y^{\theta}, v^{\theta})\delta y+p^{\theta}(v^{\phi}-v^{\theta})dsdt}_{(E)}\\
    		&+\int_{\Omega_T}F(y^{\theta}, u^{\phi})-F(y^{\theta}, u^{\theta})-p^{\theta}[f(y^{\theta}, u^{\phi})-f(y^{\theta},u^{\theta})]\\
    		&+F'_y(y^{\phi}, u^{\theta})\delta y-F'_y(y^{\theta}, u^{\theta})\delta y-[f'_y(y^{\theta}, u^{\phi})-f'_y(y^{\theta},u^{\theta})]p^{\theta}\delta y\\
    		&+\frac{1}{2}F'_{yy}(\xi_2,u^{\phi})(\delta y)^2-\frac{1}{2}p^{\theta}F'_{yy}(\xi_3,u^{\phi})(\delta y)^2dxdt.  		
    	\end{align*}
    	Using the definition of the Hamiltonian, Lemma 3.1 and Hölder's inequality, we obtain the following estimates
    	\begin{align*}
    	(A)+(D)&\leq(E)+\int_{\Omega_T}[H_{\Omega_T}(y^{\theta}, u^{\phi})-H_{\Omega_T}(y^{\theta}, u^{\theta})]\\
    	&+[\nabla_yH_{\Omega_T}(y^{\theta}, u^{\phi})-\nabla_yH_{\Omega_T}(y^{\theta}, u^{\theta})]\delta ydxdt+C\parallel\delta y\parallel^2_{L^2(\Omega_T)}.    		
    	\end{align*}
    Using a similar approach as above, we can derive an estimate for \( (B) + (E) \)
    	\begin{align*}
    	(B)+(E)&=\int_{\Sigma_T}G(y^{\phi}, v^{\phi})-G(y^{\theta}, v^{\theta})-G'_y(y^{\theta}, v^{\theta})\delta y-p^{\theta}(v^{\phi}-v^{\theta})dsdt\\
    	&=\int_{\Sigma_T}G(y^{\theta}, v^{\phi})-G(y^{\theta}, v^{\theta})-p^{\theta}(v^{\phi}-v^{\theta})\\
    	&+G'_y(y^{\theta}, v^{\phi})\delta y-G'_y(y^{\theta}, v^{\theta})\delta y+\frac{1}{2}G'_{yy}(\xi_4, v^{\theta})(\delta y)^2dsdt\\
    	&\leq \int_{\Sigma_T}[H_{\Sigma_T}(y^{\theta}, u^{\phi})-H_{\Sigma_T}(y^{\theta}, u^{\theta})]\\
    	&+[\nabla_yH_{\Sigma_T}(y^{\theta}, u^{\phi})-\nabla_yH_{\Sigma_T}(y^{\theta}, u^{\theta})]\delta ydsdt+C\parallel\delta y\parallel^2_{L^2(\Sigma_T)}.    		
    	\end{align*}
    	Substituting the estimates of \( A + B + C \) into \eqref{8}, and applying Young's inequality and Lemma 3.2, we derive
    	\begin{align*}
    	&J(u^{\phi},v^{\phi})-J(u^{\theta},v^{\theta})\leq\int_{\Omega_T}[H_{\Omega_T}(y^{\theta}, u^{\phi})-H_{\Omega_T}(y^{\theta}, u^{\theta})]\\
    	&+[\nabla_yH_{\Omega_T}(y^{\theta}, u^{\phi})-\nabla_yH_{\Omega_T}(y^{\theta}, u^{\theta})]\delta ydxdt\\
    	&+\int_{\Sigma_T}[H_{\Sigma_T}(y^{\theta}, u^{\phi})-H_{\Sigma_T}(y^{\theta}, u^{\theta})]\\
    	&+[\nabla_yH_{\Sigma_T}(y^{\theta}, u^{\phi})-\nabla_yH_{\Sigma_T}(y^{\theta}, u^{\theta})]\delta ydsdt\\
    	&+C\big(\parallel\delta y\parallel^2_{L^2(\Omega_T)}+\parallel\delta y\parallel^2_{L^2(\Sigma_T)}\big)\\
    	&\leq\int_{\Omega_T}[H_{\Omega_T}(y^{\theta}, u^{\phi})-H_{\Omega_T}(y^{\theta}, u^{\theta})]dxdt\\
    	&+\int_{\Sigma_T}[H_{\Sigma_T}(y^{\theta}, u^{\phi})-H_{\Sigma_T}(y^{\theta}, u^{\theta})]dsdt\\
    	&+\frac{1}{2}\parallel\nabla_yH_{\Omega_T}(y^{\theta}, u^{\phi})-\nabla_yH_{\Omega_T}(y^{\theta}, u^{\theta})\parallel^2_{L^2(\Omega_T)}\\
    	&+\frac{1}{2}\parallel\nabla_yH_{\Sigma_T}(y^{\theta}, u^{\phi})-\nabla_yH_{\Sigma_T}(y^{\theta}, u^{\theta})\parallel^2_{L^2(\Sigma_T)}\\
    	&+C\big(\parallel\delta y\parallel^2_{L^2(\Omega_T)}+\parallel\delta y\parallel^2_{L^2(\Sigma_T)}\big)\\
    	&\leq\int_{\Omega_T}[H_{\Omega_T}(y^{\theta}, u^{\phi})-H_{\Omega_T}(y^{\theta}, u^{\theta})]dxdt\\
    	&+\int_{\Sigma_T}[H_{\Sigma_T}(y^{\theta}, u^{\phi})-H_{\Sigma_T}(y^{\theta}, u^{\theta})]dsdt\\
    	&+C\big(\parallel\nabla_yH_{\Omega_T}(y^{\theta}, u^{\phi})-\nabla_yH_{\Omega_T}(y^{\theta}, u^{\theta})\parallel^2_{L^2(\Omega_T)}\\
    	&+\parallel\nabla_pH_{\Omega_T}(y^{\theta}, u^{\phi})-\nabla_pH_{\Omega_T}(y^{\theta}, u^{\theta})\parallel^2_{L^2(\Omega_T)}\\
    	&+\parallel\nabla_yH_{\Sigma_T}(y^{\theta}, u^{\phi})-\nabla_yH_{\Sigma_T}(y^{\theta}, u^{\theta})\parallel^2_{L^2(\Sigma_T)}\\
    	&\parallel\nabla_pH_{\Sigma_T}(y^{\theta}, u^{\phi})-\nabla_pH_{\Sigma_T}(y^{\theta}, u^{\theta})\parallel^2_{L^2(\Sigma_T)}
    	\big).
    	\end{align*}
    	Using Lemma 3.1, Lemma 3.2, and Hölder's inequality, we derive the final estimate
    		\begin{align*}
    				J(u^{\phi},v^{\phi})-J(u^{\theta},v^{\theta})&\leq\int_{\Omega_T}[H_{\Omega_T}(y^{\theta}, u^{\phi})-H_{\Omega_T}(y^{\theta}, u^{\theta})]dxdt\\
    			&+\int_{\Sigma_T}[H_{\Sigma_T}(y^{\theta}, u^{\phi})-H_{\Sigma_T}(y^{\theta}, u^{\theta})]dsdt\\
    			&+\tilde{C}\big(\parallel u^{\phi}-u^{\theta}\parallel^2_{L^2(\Omega_T)}+\parallel v^{\phi}-v^{\theta}\parallel^2_{L^2(\Sigma_T)}\big).
    		\end{align*}
    \end{proof}
    From Lemma 3.3, it can be observed that Algorithm \ref{alg:msa} replaces the minimization of the overall control index \( J \) with the minimization of the Hamiltonian. However, if the control gap between the updated control and the control generated in the previous iteration is too large, the overall control index \( J \) may fail to decrease. To address this issue, it is necessary to modify the Hamiltonian optimization step in Algorithm \ref{alg:msa} to ensure that \( J \) decreases after each iteration.
    \subsection{Augmented MSA}
    To guarantee the decrease of the overall control index \( J \), it is crucial to regulate the gap between the updated control variable and the control variable from the previous iteration. Inspired by the augmented Lagrangian method, a penalty term is incorporated into the original Hamiltonian. For a fixed penalty factor \( \rho > 0 \), we define an augmented distributed Hamiltonian function and an augmented boundary Hamiltonian function as follows:
    \begin{equation}
    	\begin{aligned}
    		\tilde{H}_{\Omega_T}(x,t,y,u,p,m):=H_{\Omega_T}(x,t,y,u,p)+\rho(m-u)^2,
    	\end{aligned}
    \end{equation}
        \begin{equation}
    	\begin{aligned}
    		\tilde{H}_{\Sigma_T}(x,t,y,v,p,n):=H_{\Sigma_T}(x,t,y,v,p)+\rho(n-v)^2.
    	\end{aligned}
    \end{equation}
    Next, we present the corresponding first-order necessary condition, referred to as the augmented Pontryagin's principle.
    \begin{proposition}
Suppose \( (\bar{y}, \bar{u}, \bar{v}) \) is the solution of \( (P) \). Then, there exists a unique adjoint state \( \bar{p} \in \mathcal{W}(0,T;L^2(\Omega),H^{1}(\Omega))\cap C(\Omega_T) \), satisfying equation \eqref{3}, such that  
\begin{align*}  
	& \tilde{H}_{\Omega_T}(x,t,\bar{y},\bar{u},\bar{p},\bar{u}) = \min_{u \in U_{ad}} \tilde{H}_{\Omega_T}(x,t,\bar{y},u,\bar{p},\bar{u}), \\  
	& \mathrm{for \; a.e.} \; (x,t) \in \Omega_T, \\  
	& \tilde{H}_{\Sigma_T}(s,t,\bar{y},\bar{v},\bar{p},\bar{v}) = \min_{v \in V_{ad}} \tilde{H}_{\Sigma_T}(s,t,\bar{y},v,\bar{p},\bar{v}), \\  
	& \mathrm{for \; a.e.} \; (s,t) \in \Sigma_T.  
\end{align*}  
    \end{proposition}
\begin{proof}The following results can be readily derived using the Pontryagin's principle
	\begin{align*}
	&\tilde{H}_{\Omega_T}(x,t,\bar{y},\bar{u},\bar{p},\bar{u})=H_{\Omega_T}(x,t,\bar{y},\bar{u},\bar{p})\leq H_{\Omega_T}(x,t,\bar{y},u,\bar{p})\\
	&\leq H_{\Omega_T}(x,t,\bar{y},u,\bar{p})+\rho(\bar{u}-u)^2=\tilde{H}_{\Omega_T}(x,t,\bar{y},u,\bar{p},\bar{u}).
		\end{align*}
\end{proof}
Based on the above proposition, we derive the augmented MSA as presented below.
       \begin{algorithm}[H]
       	\caption{Augmented MSA}
       	\label{alg:amsa}
       	\begin{algorithmic}[1]
       		\State \textbf{Initialize:} Choose initial controls $(u_{1}, v_{1}) \in L^r(\Omega_T) \times L^q(\Sigma_T)$, set penalty factor $\rho$, set iteration index $i = 1$. 
       		\Repeat 
       		\State \textbf{Solve the state equation:}
       		\[
       		\frac{\partial y_i}{\partial t} + A y_i + f(x, t, y_i, u_i) = 0, \quad 
       		\frac{\partial y_i}{\partial n_A} + v_i = 0, \quad y_i(\cdot,0) = y_0.
       		\]
       		\State \textbf{Solve the adjoint equation:}
       		\[
       		-\frac{\partial p_i}{\partial t} + A p_i = \nabla_y H_{\Omega_T}(x, t, y_i, u_i, p_i),
       		\]
       		\[
       		\frac{\partial p_i}{\partial n_A} = \nabla_y H_{\Sigma_T}(s, t, y_i, v_i, p_i), \quad 
       		p_i(\cdot, T) = L'_y(x, y_i(\cdot, T)).
       		\]
       		\State \textbf{Update the controls:}
       		\[
       		\begin{aligned}
       			u_{i+1} &= \arg\min_{u \in U_{ad}} \tilde{H}_{\Omega_T}(x, t, y_i, u, p_i,u_i), \\
       			v_{i+1} &= \arg\min_{v \in V_{ad}} \tilde{H}_{\Sigma_T}(s, t, y_i, v, p_i,v_i),
       		\end{aligned}
       		\]
       		for each $(x, s, t) \in \Omega \times \Gamma \times [0,T]$.
       		\State Increment the iteration index: $i := i + 1$.
       		\Until{The termination criterion is satisfied.}
       	\end{algorithmic}
       \end{algorithm}
       In Algorithm \ref{alg:amsa}, the Hamiltonian minimization step in the MSA is replaced with the minimization of the augmented Hamiltonian. By selecting an appropriate value for the penalty factor \( \rho \), the control index \( J \) is guaranteed to decrease with each iteration.   
    \subsection{Convergence of algorithm}
    We now establish the convergence of Algorithm \ref{alg:amsa} through the following theorem.
    \begin{theorem}
Suppose Assumption 1 and Assumption 2 hold, and let \( u_0 \in \mathcal{U} \) and \( v_0 \in \mathcal{V} \) be any initial measurable controls satisfying \( J(u_0, v_0) < +\infty \). Further, assume that \( \inf_{(u,v) \in \mathcal{U} \times \mathcal{V}} J(u,v) > -\infty \). Then, for \( \rho > \tilde{C} \), if \( u_i \) and \( v_i \) are generated by Algorithm \ref{alg:amsa}, there exist \( \tilde{u} \in \mathcal{U} \) and \( \tilde{v} \in \mathcal{V} \) such that  
\[
u_k \to \tilde{u} \; \mathrm{in} \; L^2(\Omega_T) \quad \mathrm{and} \quad v_k \to \tilde{v} \; \mathrm{in} \; L^2(\Sigma_T).
\]
    \end{theorem}
    \begin{proof}According to Lemma 3.3, we have
\begin{align*}
&J(u_{i+1},v_{i+1})-J(u_i,v_i)\leq\int_{\Omega_T}[H_{\Omega_T}(y_i, u_{i+1})-H_{\Omega_T}(y_i, u_i)]dxdt\\
&+\int_{\Sigma_T}[H_{\Sigma_T}(y_i, v_{i+1})-H_{\Sigma_T}(y_i, v_i)]dsdt\\
&+\tilde{C}\big(\parallel u_{i+1}-u_i\parallel^2_{L^2(\Omega_T)}+\parallel v_{i+1}-v_i\parallel^2_{L^2(\Sigma_T)}\big).
\end{align*}
From the augmented Hamiltonian minimization step of Algorithm \ref{alg:amsa}, we obtain
\begin{align*}
H_{\Omega_T}(y_i, u_{i+1})+\rho(u_{i+1}-u_i)^2\leq H_{\Omega_T}(y_i, u_i),
\end{align*}
\begin{align*}
H_{\Sigma_T}(y_i, v_{i+1})+\rho(v_{i+1}-v_i)^2\leq H_{\Sigma_T}(y_i, v_i),
\end{align*}
which implies
\begin{align*}
&J(u_{i+1},v_{i+1})-J(u_i,v_i)\\
&\leq(\tilde{C}-\rho)\big(\parallel u_{i+1}-u_i\parallel^2_{L^2(\Omega_T)}+\parallel v_{i+1}-v_i\parallel^2_{L^2(\Sigma_T)}\big).
\end{align*}
Summing \( N \) terms under the condition \( \rho > \tilde{C} \), we derive the following estimate
\begin{align*}
0&\leq\sum_{i=0}^{N}\big(\parallel u_{i+1}-u_i\parallel^2_{L^2(\Omega_T)}+\parallel v_{i+1}-v_i\parallel^2_{L^2(\Sigma_T)}\big)\\
&\leq(\rho-\tilde{C})\big(J(u_0,v_0)-J(u_N,v_N)\big)\\
&\leq(\rho-\tilde{C})\big(J(u_0,v_0)-\inf_{(u,v)\in\mathcal{U}\times\mathcal{V}} J(u,v)\big).
\end{align*}
Thus, we have  
\[
\sum_{i=0}^{\infty} \big( \| u_{i+1} - u_i \|^2_{L^2(\Omega_T)} + \| v_{i+1} - v_i \|^2_{L^2(\Sigma_T)} \big) < +\infty.
\]  
By the Cauchy convergence criterion, there exist \( \tilde{u} \in \mathcal{U} \) and \( \tilde{v} \in \mathcal{V} \) such that \( u_i \to \tilde{u} \) in \( L^2(\Omega_T) \) and \( v_i \to \tilde{v} \) in \( L^2(\Sigma_T) \).
    \end{proof}    
    \section{Numerical tests}
In this section, we present numerical results for an optimal control problem involving a parabolic equation, where \( \Omega \) is a two-dimensional domain. The problem \eqref{P} is solved using the augmented Method of Successive Approximations (MSA) as described in Algorithm \ref{alg:amsa}. The termination condition for the algorithm is given by
\[
J(u_{i+1}) - J(u_i) < \epsilon.
\]
We consider the following optimal control problem with \( \Omega_T = [0,1] \times [0,1] \times [0,1] \):
\[
\min J(y,u) := \frac{1}{2} \| y(\cdot,T) - y_d \|_{L^2(\Omega)}^2 + \frac{\alpha}{2} \| u \|_{L^2(\Omega_T)}^2,
\]
subject to
\[
\begin{cases}
	y_t - \Delta y = u + y, & \mathrm{in }\; \Omega_T, \\
	\partial_{\nu_A} y = 0, & \mathrm{on }\; \Sigma_T, \\
	y(\cdot, 0) = y_0, & \mathrm{in }\; \overline{\Omega}.
\end{cases}
\]
Here, we choose the following parameters: \( u_0 = 0.01 \), \( y_0 = \sin(\pi x) \sin(\pi y) \), \( y_d = \exp(-2\alpha\pi T) \sin(\pi x) \sin(\pi y) \), \( \alpha = 1 \), and \( \rho = 1 \). 

To solve the state and adjoint equations, we employ the finite difference method. The spatial step size is set to \( 0.01 \) in each direction, and the time step size is set to \( 0.04 \). The learning rate for the gradient descent in the sub-problem is dynamically adjusted, starting at \( 0.001 \) and multiplied by \( 0.9 \) every 100 iterations. The termination criterion is set to \( \epsilon = 10^{-4} \). The numerical results obtained are presented below.\\
    \begin{figure}[ht]
	\centering
	\includegraphics[width=0.7\textwidth]{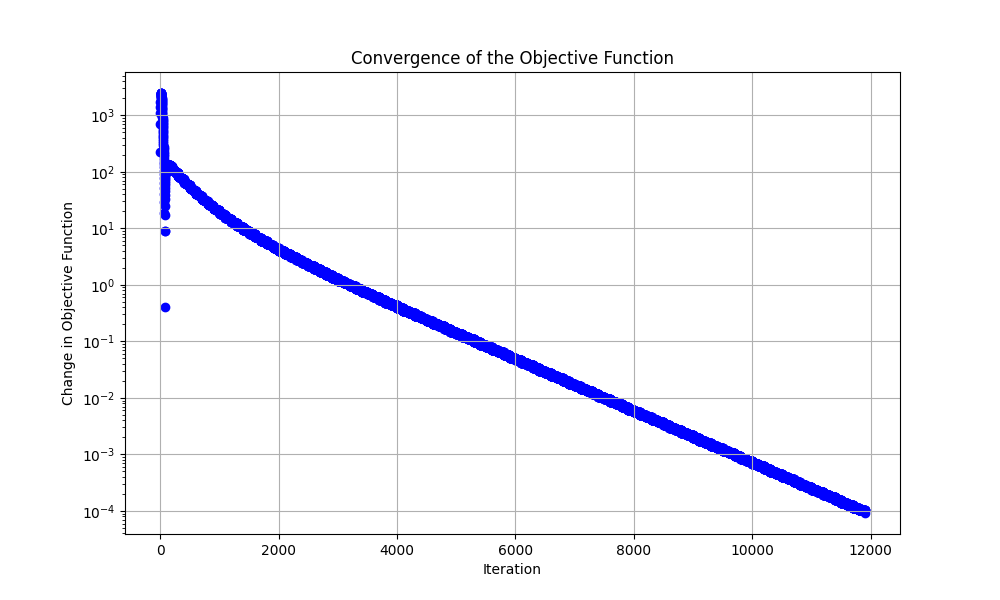} % 使用相对路径
	\caption{The difference in the objective function $J(u_{i+1})-J(u_i)$} 
	\label{fg1} 
\end{figure}
\begin{figure}[ht]
	\centering
	\includegraphics[width=0.7\textwidth]{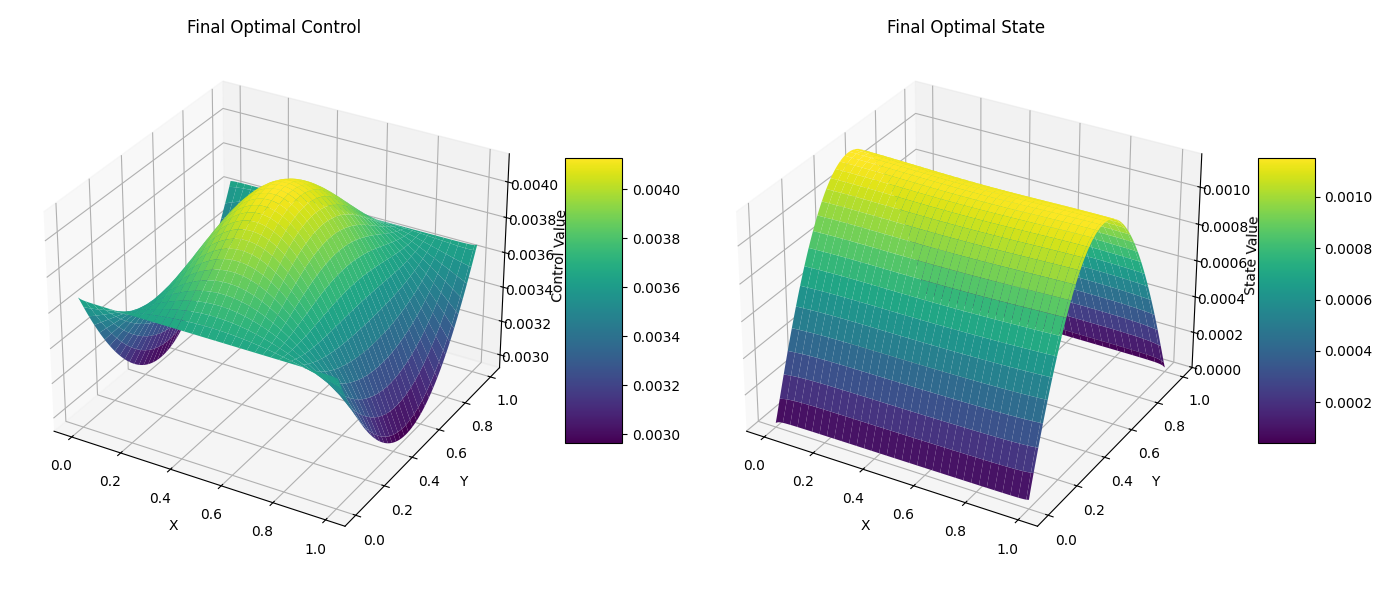} % 使用相对路径
	\caption{ Computed discrete optimal state $y$ (right) and optimal control $u$ (left)} 
	\label{fg2}
\end{figure}
\\
As shown in \ref{fg1}, the difference in objective function values generally exhibits a monotonically decreasing trend, demonstrating the effectiveness of the augmented MSA (AMSA) presented in this paper. Since gradient descent is employed for solving the Hamiltonian optimization step, the algorithm may converge to a local optimal solution. This limitation can be addressed by incorporating a global optimization algorithm for further improvement.

    \bibliographystyle{plain}
    \bibliography{AMSA}
\end{document}